\documentclass[12pt]{amsart}
\usepackage{amsmath}
\usepackage{amssymb}
\usepackage{amsfonts}
\usepackage{amsthm}
\usepackage{mathrsfs}
\usepackage[all]{xy}
\usepackage{color}

\newcommand{\be}{\begin{equation}}
\newcommand{\ee}{\end{equation}}
\newtheorem{theorem}{Theorem}[section]
\newtheorem{lemma}[theorem]{Lemma}

\newtheorem{proposition}[theorem]{Proposition}
\newtheorem{corollary}[theorem]{Corollary}

\theoremstyle{definition}
\newtheorem{definition}[theorem]{Definition}

\newtheorem{remark}[theorem]{Remark}

\numberwithin{equation}{section}

\begin{document}

\title{Generalized Bell states and principal realization of the Yangian $Y(\mathfrak{sl}_{N})$}
\author{Ming Liu, Chengming Bai, Mo-Lin Ge, Naihuan Jing$^*$}
\address{ML: Chern Institute of Mathematics, Nankai University, Tianjin 300071, China}
\email{ming.l1984@gmail.com}
\address{CB: Chern Institute of Mathematics and LPMC, Nankai University, Tianjin 300071, China}
\email{baicm@nankai.edu.cn}
\address{MG: Chern Institute of Mathematics, Nankai University, Tianjin 300071, China}
\email{geml@nankai.edu.cn}
\address{NJ: Department of Mathematics, North Carolina State University, Raleigh, NC 27695, USA}
\email{jing@math.ncsu.edu}

\thanks{{\scriptsize
\hskip -0.4 true cm MSC (2010): Primary: 17B37; Secondary: 17B65, 81P68
$*$Corresponding author, jing@math.ncsu.edu
}}

\maketitle

\begin{abstract} We prove that the action of the Yangian algebra $Y(\mathfrak{sl}_N)$ is better described by
the principal generators on the tensor product of
the fundamental representation and its
dual. The generalized Bell states or
maximally entangled states are permuted by the principal generators in a dramatically simple manner on the tensor
product. Under the Yangian symmetry the new quantum number
$\bf J^2$ is also explicitly computed, which gives an explanation for these maximally entangled states.
\end{abstract}

\section{Introduction}

Let $Y(\mathfrak{sl}_N)$ be the Yangian algebra associated to the simple Lie algebra
$\mathfrak{sl}_N$. The algebra was first introduced by Drinfeld in his study of
rational Yang-Baxter equation. Drinfeld's main idea \cite{D1, D2, D3} was that the Yangian algebra has
a noncommutative and noncocommutative Hopf algebra structure and that
the finite dimensional representations of $Y(\mathfrak{sl}_N)$ are parametrized by Drinfeld polynomials.
The structure of the Yangian algebra is
governed by two sets of generators: the usual Lie algebra generators $x$ and
the associated Yangian generators $J(x)$ \cite{D1} . In general the action of the Yangian generators $J(x)$ is
quite complicated compared with the Serre generators. A natural question is finding a suitable formula for $J(x)$ or
representing them in a nice representation.

On the other hand, entangled states play important roles in quantum computation and quantum information \cite{Bo,BPD,GHSZ,GHZ,KOC}.
It is known that the $SU(N)$ entangled states with the maximal degree of entanglement
constitute a special basis of the tensor
representation of the fundamental representation and its dual representation 
\cite{KOC}.
In \cite{BGJ,GBXZ} it was observed that 
these entangled states constitute a ``natural" basis for the action of the Yangian algebra $Y(\mathfrak{sl}_3)$. Under the natural basis of the entangled states with the maximal degree of entanglement, the action of a special set of generators, namely the principal generators, of the Yangian algebra $Y(\mathfrak{sl}_3)$ becomes extraordinarily simple, and behave much like vacuum states
or extremal states.

Naively speaking, the Cartan-Weyl basis elements for the
general linear Lie algebra $\frak{gl}_N$ are the unit matrices while
the principal basis elements are Toeplitz matrices, which are
diagonal-constant. The principal generators can also be viewed as certain discrete Fourier
transform of the Cartan-Weyl basis elements \cite{BGJ,JL}. In the case of $N=3$, the principle generators
(such as the $J$-parts) of $Y(\mathfrak{sl}_3)$ behave much like shift operators permuting different entangled states, whereas obviously the Cartan-Weyl
basis does not enjoy this property. By using the explicit form in the example a conjecture was made in \cite{GBXZ} for
the general case of $Y(\mathfrak{sl}_N)$. The conjecture not only reveals the general rule for the action but
provides a practical way to understand the action of
the Yangian algebras in connection with quantum computation.

 In this paper we will prove this conjecture in complete
generality by using the newly discovered principal generators of the Yangian algebra \cite{BGJ, JL}.
Moreover we also generalize the notion of the quantum number ${\bf J}^2$ and compute its action on
the tensor product for $Y(\mathfrak{sl}_N)$. The quantum number ${\bf J}^2$ was first introduced in \cite{BGX} for $Y(\mathfrak{sl}_2)$
and played an important role in study of the spectrum of hydrogen atom and other models (\cite{BGX2,PBG}), where
it was shown that the quantum number ${\bf J}^2$ seems to provide another invariant in addition to
usual Casimir operator ${\bf I}^2$.
By a similar discussion as in \cite{GBXZ}, the quantum
 number may serve to describe the maximal degree of entanglement with special choice of parameters.

The paper is organized as follows. First in section two we recall the principal generators
for type $A$.
In section three we first recall the Drinfeld's realization of Yangian and its Hopf algebra structure in subsection \ref{sub3.1},
then in subsection \ref{sub3.2} we give the action of $Y(\mathfrak{sl}_N)$ on the tensor representation of the fundamental
representation and its dual representation, which is our main result. At last in subsection \ref{sub3.3}, we give the action of ${\bf J}^2$ explicitly.

\section{Principal realization of Yangian algebra $Y(\mathfrak{gl}_N)$}

\vskip 0.2in
   We first recall the principal realization for $Y(\mathfrak{gl}_N)$ given in
   \cite{BGJ}. The new basis elements are coefficients of the discrete Fourier
   transform of certain Toeplitz sequences associated with the abelian group $\mathbb Z_N$ \cite{JL}.
 
\subsection{The Yangian algebra associated to $\mathfrak{gl}_N$}
Fix a natural number $N$. Let $\mathfrak{g}=\mathfrak{gl}_N$ be the general linear Lie algebra.
\begin{definition}
The Yangian algebra $Y(\mathfrak{gl}_N)$ is an unital associative algebra
with generators $t_{ij}^{(r)}$ ($i,j\in \{1, 2, \ldots, N\}, r\in \mathbb{N}$)
 subject to the relations:
\begin{align}\label{definingrelation1}
[t^{(r+1)}_{ij},t^{(s)}_{kl}]-[t^{(r)}_{ij},t^{(s+1)}_{kl}]
=t^{(r)}_{kj}t^{(s)}_{il}-t^{(s)}_{kj}t^{(r)}_{il},
\end{align}
where $t_{ij}^{(0)}=\delta_{ij}$.
\end{definition}
The general linear Lie algebra $\mathfrak{gl}_N$ is a subalgebra of $Y(\mathfrak{gl}_N)$: as
taking $r=0, s=1$, we see that $t_{ij}^{(1)}$ are just like the unit matrix elements $E_{ij}$:
\be
[t_{ij}^{(1)}, t_{kl}^{(1)}]=\delta_{kj}t_{ij}^{(1)}-\delta_{il}t_{kj}^{(1)}.
\ee
\vskip .1in

In the standard triangular decomposition of $\mathfrak{gl}_N$,
the unit matrices $E_{ij}$ consist of the Cartan-Weyl basis. The principal
basis comes from another root space decomposition given by Toeplitz matrices. Let
$$E=\sum_{i\in\mathbb{Z}_N}E_{i,i+1}$$
be the standard Toeplitz matrix. The principal Cartan subalgebra is spanned by the centralizer $C(E)$,
and the principal root generators are
\begin{equation}\label{fomula2.1}
A_{ij}=\sum_{k\in \mathbb{Z}_N}\omega^{ki}E_{k,k+j},
\end{equation}
where $i,j \in \mathbb{Z}_{N}$, $\omega=e^{\frac{\mathfrak{i}2\pi}{N}}$.
Under the standard bilinear
form $(x|y)=tr(xy)$, the principal basis $\{A_{ij}\}$ is dual to $\{\frac{\omega^{ij}}{N}A_{-i,-j}\}$. This
also follows from the algebraic properties:
\begin{equation*}
A_{ij}A_{kl}=\omega^{jk}A_{i+k,j+l}.
\end{equation*}

Let us define matrix-valued sequences indexed by $\mathbb Z_N$. Fix $j\in \mathbb Z_N$, the Toeplitz sequence $\{\epsilon_j\}$ is defined by $\epsilon_j(k)=E_{k,k+j}$, where $k$ is the index variable. Then the principal basis elements $A_{ij}$ can be written as the discrete Fourier transform:
$$A_{ij}=\mathfrak{F}(\{\epsilon_j\})(i)=\sum_{k\in \mathbb{Z}_N}\omega^{ki}E_{k,k+j}.$$
\vskip .2in

Therefore the usual Cartan-Weyl basis is simply given by the inverse Fourier transform (see \cite{JL}):
\be
E_{k,k+j}=\epsilon_j(k)=\frac1N\sum_{l=0}^{N-1}\omega^{-kl}A_{lj}.
\ee

To introduce the principal basis for the Yangian $Y(\mathfrak{gl}_N)$, we recall its matrix formulation via
Yang-Baxter $R$-matrix \cite{M, MNO}.
 Let $u$ be a formal variable and let
 \begin{align}\label{R-matrix}
 R(u)=1-\frac{P}{u}\in {\rm End}(\mathbb{C}^{N})\otimes {\rm End}( \mathbb{C}^{N})[[u^{-1}]],
 \end{align}
where $P$ is the permutation matrix: $P(u\otimes v)=v\otimes u$ for any $u, v\in \mathbb{C}^N$.
The matrix $R(u)$ satisfies the quantum Yang-Baxter equation:
 $$R_{12}(u)R_{13}(u+v)R_{23}(v)=R_{23}(v)R_{13}(u+v)R_{12}(u).$$

Set $$T(u)=\sum_{i,j}t_{ij}(u)\otimes E_{ij}\in
Y(\mathfrak{gl}_N)[[u^{-1}]]\otimes {\rm End}(\mathbb{C}^{N}),$$ where
 $$t_{ij}(u)=\delta_{ij}+\sum_{k=1}^{\infty}t_{ij}^{(k)}u^{-k}\in Y(\mathfrak{gl}_N)[[u^{-1}]].$$
It is well-known that the defining relations of Yangian can be written compactly as (cf. \cite{M})
 $$R(u-v)T_{1}(u)T_{2}(v)=T_{2}(v)T_{1}(u)R(u-v),$$
 and the coproduct is given by
 $$\triangle(T)=T\otimes T.$$

\subsection{The principal realization of $Y(\mathfrak{gl}_N)$}
For $i,j\in \mathbb{Z}_{N}$, let $s_{ij}(u)$ be the generating series of the principal generators \cite{BGJ}:
 $$s_{ij}(u)=\sum_{n=0}^{\infty}s_{ij}^{(n)}u^{-n},$$
 where $s_{ij}^{(0)}=\delta_{i,0}\delta_{j,0}$. Then
 \begin{equation}\label{prin1}
 s_{ij}(u)=\sum_{k\in\mathbb Z_N}\frac{\omega^{-ki}}{N}t_{k,j+k}(u).
 \end{equation}
This is in fact the inverse of the Fourier transform of the finite sequence $\{t_{k,k+j}(u)\}$, where
the variable $k\in \mathbb{Z}_N$. That is, the formula (\ref{prin1}) can be rewritten as \cite{JL}
$$s_{ij}=\mathfrak{F}^{-1}(\{t_{k,k+j}(u)\})(i).$$
 Rewriting the T-matrix $T(u)$ by using the principal basis
 of $\mathfrak{gl}_N$ and $s_{ij}(u)$ as follows
 $$T(u)=\sum_{k,l\in \mathbb{Z}_N}s_{kl}(u)\otimes A_{kl},$$
 we obtain the principal realization of $Y(\mathfrak{gl}_N)$ as follows.
 \vskip.2in

The principal generators also satisfy some quadratic matrix equations.
\begin{theorem} \cite{JL} \label{new}
 The principal generators $s_{ij}^{(k)}$ of Yangian $Y(\mathfrak{gl}_N)$ satisfy the following relations:
 \begin{align*}
 &(u-v)[s_{ij}(u),s_{kl}(v)]=\frac1N\sum_{a,b}\omega^{-ab}(s_{k+a,j+b}(u)s_{i-a,j-b}(v)-s_{k+a,j+b}(v)s_{i-a,j-b}(u)),
 \end{align*}
 where $a,b$ run through the group $\mathbb Z_N$.
\end{theorem}
It is not hard to prove that these relations are defining relations of the principal generators.
 In fact the defining relations given in \cite{BGJ} are consequences of Theorem \ref{new}.

\section{Representations of the Yangian $Y(\mathfrak{sl}_{N})$}

In this section, we will show that the principal basis of $\mathfrak{sl}_N$ takes a simple form in certain representation
of Yangian $Y(\mathfrak{sl}_{N})$ in association with the study of entangled states in quantum information.
\subsection{Drinfeld's definition for Yangian}\label{sub3.1}

To study representations of the Yangian $Y(\mathfrak{sl}_{N})$, we need Drinfeld's definition of Yangian and its
Hopf structure.

\begin{definition} (Drinfeld)
Let $\mathfrak{g}$ be a finite dimensional simple Lie algebra equipped with a non-degenerate invariant bilinear
form $(,)$,
and let $\{I_{\lambda}\}$ be an orthonormal basis with respect to $(\, ,\, )$. The Yangian algebra $Y(\mathfrak{g})$ associated with $\mathfrak{g}$
is an associative algebra generated by elements $x$, $J(x)$, for $x\in\mathfrak{g}$, with the following defining relations:
\begin{equation}
[x,y]~~~(in ~~~ Y(\mathfrak{g}))=[x,y]~~~ (in ~~~ \mathfrak{g}),
\end{equation}

\begin{equation}
J(ax+by)=aJ(x)+bJ(y),
\end{equation}

\begin{equation}
[x,J(y)]=J([x,y]),
\end{equation}

\begin{equation}
\begin{aligned}
&[J(x),J([y,z])]+[J(z),J([x,y])]+[J(y),J([z,x])]=\\
&\sum_{\lambda, \mu, \nu}([x,x_{\lambda}],[[y,x_\mu],[z,x_\nu]])\{x_{\lambda},x_\mu,x_\nu\},\\
\end{aligned}
\end{equation}

\begin{equation}
\begin{aligned}
&[[J(x),J(y)],[z,J(w)]]+[[J(z),J(w)],[x,J(y)]]=\\
&\sum_{\lambda, \mu, \nu}([x,x_{\lambda}],[[y,x_\mu],[[z,w],x_\nu]])\{x_{\lambda},x_\mu,J(x_\nu\}),\\
\end{aligned}
\end{equation}
for all $x,y,z,w \in\mathfrak{g}$, $a,b\in \mathbb{C}$. Here $\{z_1,z_2,z_3\}=\frac{1}{24}\sum_{\pi} z_{\pi(1)}z_{\pi(2)}z_{\pi(3)}$
the sum runs through all permutations $\pi$ of $\{1,2,3\}$.

\vskip.1in
The Hopf structure of $Y(\mathfrak{g})$ is given by the coproduct $\Delta$ and the antipode $S$:

\begin{equation}
\Delta(x)=1\otimes x+x\otimes 1,
\end{equation}

\begin{equation}\label{eq3.7}
\Delta(J(x))=1\otimes J(x)+J(x)\otimes 1+\frac{1}{2}[x\otimes 1,t],
\end{equation}

\begin{equation}
S(x)=-x,  S(J(x))=J(x)+\frac{1}{4}cx,
\end{equation}

\begin{equation}
\epsilon(x)=\epsilon(J(x))=0,
\end{equation}
where $c$ is the eigenvalue of the Casimir element $t\in U(\mathfrak{g})$ in the adjoint representation of $\mathfrak{g}$.

\end{definition}

From the last section, we know that the principal basis consists of $\{A_{ij}\}$ 
  and its dual basis is $\{\frac{\omega^{ij}}{N}A_{-i,-j}\}$
  under the bilinear form $(x,y)=tr(xy)$. Here $(i,j)\in \mathbb{Z}_{N}^{2}\setminus\{0,0\}$. Thus the Casimir element $t$ in $\mathfrak{sl}_N$
 can be written as follows:
 \begin{equation}\label{eq3.10}
 t=\sum_{(i,j)\in \mathbb{Z}_{N}^{2}\setminus\{0,0\}}\frac{\omega^{ij}}{N}A_{ij}\otimes A_{-i,-j} \in U(\mathfrak{sl}_N)^{\otimes 2}.
 \end{equation}

For convenience we slightly modify the principal basis as follows:
\begin{equation}\label{modified principal basis}
T^{(j)}_{i}=\omega^{-i+1}A_{i-1,j-1},
 \end{equation}
where $(i,j)\in\mathbb{Z}_{N}^{2}\backslash \{(1, 1)\}$ as $T^{(1)}_{1}=I$ will be excluded in $Y(\mathfrak{sl}_N)$.

 Using the modified principal basis of $\mathfrak{sl}_{N}$, we can write down the coproduct of $J(T^{(j)}_{i})$ explicitly.
 \begin{lemma}\label{coproduct}
\begin{equation}\label{eq3.12}
 \begin{aligned}
 &\Delta(J(T^{(j)}_{i}))=(1\otimes J(T^{(j)}_{i})+J(T^{(j)}_{i})\otimes 1)+\\
 &\frac{1}{2N}\sum_{(k,l)\in \mathbb{Z}_{N}^{2}\setminus\{0,0\}}(\omega^{k(l+j-1)}-\omega^{l(k+i-1)})T^{(l+j)}_{k+i}\otimes T^{(-l+1)}_{-k+1}.
 \end{aligned}
 \end{equation}
 \end{lemma}

\subsection{The entangled states and $Y(\mathfrak{sl}_{N})$} \label{sub3.2}

In \cite{GBXZ, BGJ} the authors gave $Y(\mathfrak{sl}_3)$ as an example to show that
the principal basis plays an essential role in the representation of $Y(\mathfrak{sl}_3)$ in
close relation with entangled states in quantum information. We prove that the general case of $Y(\mathfrak{sl}_{N})$ has the same property
as conjectured in \cite{GBXZ}.

\vskip.2in
Let $\lambda_1=(1,0,...0)$ be the fundamental weight of the simple Lie algebra $\mathfrak{sl}_N$.
The irreducible representation $V(\lambda_{N-1})$ with $\lambda_{N-1}=(0,\cdots, 0,1)$ can be viewed as the dual representation of $V(\lambda_1)$. Suppose $|i\rangle_1=u_1,u_2, \ldots, u_{N}$
form a basis of $V(\lambda_1)$, and let $|i\rangle_2=u^{*}_1,u^{*}_2, \ldots, u^{*}_{N}$ be the dual basis.
Let $|i,j\rangle=|i\rangle_1\otimes |j\rangle_2$ $(i,j=1,2,\cdots, N)$ be the orthonormal basis in the tensor representation
$V(\lambda_1)\otimes V(\lambda_{N-1})$. For convenience, we choose another basis of $V(\lambda_1)\otimes V(\lambda_{N-1})$ as follows
\begin{equation}\label{eq3.13}
\Psi^{(m)}_{k}=\sum^{N}_{r=1} \omega^{(k-1)(r-1)}|r,m+r-1\rangle, (k,m=1,2...N).
\end{equation}
We remark that the general element in the basis is an entangled state with the maximal degree of entanglement \cite{KOC} in quantum information.
There elements are called {\it generalized Bell states} for any $N\geq 3$, whereas in the case $N=2$, they are exactly the Bell states.

As a representation of $\mathfrak{sl}_N$,
$V(\lambda_1)\otimes V(\lambda_{N-1})$ is decomposed into irreducible representations
$$V(\lambda_1)\otimes V(\lambda_{N-1})=V_{\rm ad}\oplus V_0,$$
where $V_{\rm ad}$ is equivalent to the adjoint representation of $\mathfrak{sl}_N$ and $V_0$ is equivalent to the 1-dimensional trivial representation of $\mathfrak{sl}_N$. Clearly
$\{ \Psi^{(m)}_k|(m,k)\ne (1,1)\}$ is a basis of $V_{\rm ad}$ and $\{ \Psi^{(1)}_{1}\}$ is a basis of $V_0$.

\vskip.2in
In general, any finite dimensional irreducible representation
of $\mathfrak{sl}_N$ can be lifted to a representation of $Y(\mathfrak{sl}_N)$. It is shown in \cite{CP2} that for any
fundamental representation
$V(\lambda_i)$, the action of $J(x)$ is just $ax$ for any $x\in \mathfrak{sl}_{N}$, where $a$ is a constant. So for any fundamental
representation $V(\lambda_i)$ of $\mathfrak{sl}_N$ and $a\in \mathbb{C}$, we can construct a representation of $Y(\mathfrak{sl}_{N})$ with the action given by
\begin{equation*}
x\mapsto x , J(x)\mapsto ax, \forall x\in \mathfrak{sl}_{N}
\end{equation*}
and the resulted Yangian module will be denoted by $V(\lambda_i,a)$. When we consider the
action of the Yangian $Y(\mathfrak{sl}_N)$ on
the tensor representation $V(\lambda_1,a)\otimes V(\lambda_{N-1},b)$, we need the coproduct in $Y(\mathfrak{sl}_N)$. Since an explicit expression of coproduct for Drinfeld's new realization (\cite{D2})
is extremely difficult, we try to use the modified principal generators $T^{(j)}_{i}$, $J(T^{(j)}_{i})$ where
$i,j=1,\cdots, N$ and $(i,j)\ne (1,1)$. Then
we get our main result as follows.

\begin{theorem}\label{th:main}
The action of $Y(\mathfrak{sl}_N)$ on the tensor representation $V(\lambda_1,a)\otimes V(\lambda_{N-1},b)$ can be expressed in the action of the
principal basis $J(T_{i}^{(j)})$ on the maximally entangled states $\Psi^{(m)}_{k}$ $(k,m=1,2,\cdots, N)$ in a simple way. The explicit action is given by the
following equation:
\begin{align}\nonumber
J(T_{i}^{(j)})\Psi^{(m)}_{k}&=[a\omega^{(j-1)(k-1)}-b\omega^{(i-1)(m-1)}\\
&+\frac{N}{2}(\delta_{i+k,2}\delta_{j+m,2}-\delta_{k,1}\delta_{m,1})\omega^{(j-1)(k-1)}]\Psi^{(m+j-1)}_{k+i-1},\label{eq:main}
\end{align}
for any $i,j,k,m=1,\cdots, N, (i,j)\ne (1,1)$.
\end{theorem}

\begin{remark}
{\rm The above result was first conjectured in \cite{GBXZ} in a slightly different form. Our formula (\ref{eq:main}) is in fact an equivalent expression.}
\end{remark}

\begin{remark}
{\rm The actions of $J$-parts of $Y(\mathfrak{sl}_N)$ in terms of principal generators behave much like shift operators permuting maximally entangled states or the generalized Bell states, whereas the actions of the Lie algebra $\mathfrak{sl}_N$ do not enjoy this property. This means that the Yangian actions single out the maximally entangled states. We do not know yet the meaning of the spectral parameter in quantum computation.}
\end{remark}

Before we prove the theorem we need the following lemmas.
\begin{lemma}\label{lemma3.4}
On the $\mathfrak{sl}_N$-modules $V(\lambda_1)$ and $V(\lambda_{N-1})$, we have
\begin{equation*}
T^{(j)}_{i}|m\rangle_1= \omega^{(i-1)(m-j)}|m-j+1\rangle_1,
\end{equation*}
\begin{equation*}
T^{(j)}_{i}|m\rangle_2= -\omega^{(i-1)(m-1)}|m+j-1\rangle_2.
\end{equation*}
\end{lemma}

\begin{proof}
From Eq. (\ref{fomula2.1}), we have
\begin{equation*}
T^{(j)}_{i}= \sum_{k}\omega^{(i-1)(k-1)}E_{k,k+j-1}.
\end{equation*}
So on $V(\lambda_1)$, $T^{(j)}_{i}$ acts on the basis elements $|m\rangle_1$ of $V(\lambda_1)$ as follows.
\begin{equation*}
\begin{aligned}
T^{(j)}_{i}|m\rangle_1&= \sum_{k}\omega^{(i-1)(k-1)}E_{k,k+j-1}|m\rangle_1\\
&=\sum_{k}\omega^{(i-1)(k-1)}\delta_{k,m-j+1}|k\rangle_1=\omega^{(i-1)(m-j)}|m-j+1\rangle_1.\\
\end{aligned}
\end{equation*}
Since $V(\lambda_{N-1})$ can be viewed as the dual representation of $V(\lambda_{1})$,
we can derive the action of $T^{(j)}_{i}$ on the basis elements $|m\rangle_2$ similarly.
\end{proof}

\begin{lemma}\label{lemma3.5}
In the representation $V(\lambda_1,a)\otimes V(\lambda_{N-1},b)$ of $Y(\mathfrak{sl}_N)$, we have
\begin{equation}
(J(T^{(j)}_{i})\otimes 1)\Psi^{(m)}_{k}=a\omega^{(j-1)(k-1)}\Psi^{(m+j-1)}_{k+i-1},
\end{equation}

\begin{equation}\label{equation3.15}
(1\otimes J(T^{(j)}_{i}))\Psi^{(m)}_{k}=-b\omega^{(m-1)(i-1)}\Psi^{(m+j-1)}_{k+i-1}.
\end{equation}
\end{lemma}

\begin{proof} These two formulae are proved in exactly the same way, so we only give the first one.
On the tensor product $V(\lambda_1,a)\otimes V(\lambda_{N-1},b)$ we have
\begin{equation*}
\begin{aligned}
(J(T^{(j)}_{i})\otimes 1)\Psi^{(m)}_{k}=\sum^{N}_{r=1}\omega^{(k-1)(r-1)}J(T^{(j)}_{i})|r\rangle_1\otimes |m+r-1\rangle_2.
\end{aligned}
\end{equation*}
Using Lemma \ref{lemma3.4} we can get
\begin{equation}\label{eq3.16}
\begin{aligned}
(J(T^{(j)}_{i})\otimes 1)\Psi^{(m)}_{k}
=a\sum^{N}_{r=1}\omega^{(k-1)(r-1)+(i-1)(r-j)}|r-j+1\rangle_1\otimes |m+r-1\rangle_2.
\end{aligned}
\end{equation}
The right hand side of Eq. (\ref{eq3.16}) is equal to
$$a\omega^{(k-1)(j-1)}\sum^{N}_{r=1}\omega^{(k+i-2)(r-1)}|r\rangle_1\otimes |m+j+r-2\rangle_2,$$
which is just $a\omega^{(j-1)(k-1)}\Psi^{(m+j-1)}_{k+i-1}$.
\end{proof}

\begin{lemma}\label{lemma3.6}
\begin{equation*}
\begin{aligned}
&\frac{1}{2N}\sum_{(k',l')\in \mathbb{Z}_{N}^{2}\setminus\{0,0\}}(\omega^{k'(l'+j-1)}-\omega^{l'(k'+i-1)})(T^{(l'+j)}_{k'+i}\otimes T^{(-l'+1)}_{-k'+1})\Psi_{k}^{(m)}\\
&=\frac{N}{2}\omega^{(j-1)(k-1)}(\delta_{i+k-1,1}\delta_{j+m-1,1}-\delta_{k,1}\delta_{m,1})\Psi^{(m+j-1)}_{k+i-1}
\end{aligned}
\end{equation*}
\end{lemma}

\begin{proof}
Using Eq. (\ref{eq3.13}), we obtain the following equation:
\begin{equation}
\begin{aligned}\label{eq3.17}
&(T^{(l'+j)}_{k'+i}\otimes T^{(-l'+1)}_{-k'+1})\Psi^{(m)}_{k}=\sum^{N}_{r=1} \omega^{(k-1)(r-1)}T^{(l'+j)}_{k'+i}|r\rangle_1\otimes T^{(-l'+1)}_{-k'+1})|m+r-1\rangle_2.\\
\end{aligned}
\end{equation}
By Lemma \ref{lemma3.4} we see that the right hand side of Eq. (\ref{eq3.17}) is equal to
\begin{equation*}
-\omega^{(k-k'-1)(j+l'-1)-k'(m-1)}\sum_{r=1}^{N}\omega^{(i+k-2)(r-1)}|r,r+m+j-2\rangle,
\end{equation*}
and this is $-\omega^{(k-k'-1)(j+l'-1)-k'(m-1)}\Phi_{i+k-1}^{(m+j-1)}$.
So we get the following equation:
\begin{equation}
\omega^{k'(l'+j-1)}T^{(l'+j)}_{k'+i}\otimes T^{(-l'+1)}_{-k'+1})\Psi^{(m)}_{k}=-\omega^{(k-1)(j-1)}\cdot\omega^{(k-1)l'+(1-m)k'}\Psi_{i+k-1}^{(m+j-1)},
\end{equation}

\begin{equation}
\omega^{l'(k'+i-1)}T^{(l'+j)}_{k'+i}\otimes T^{(-l'+1)}_{-k'+1})\Psi^{(m)}_{k}=-\omega^{(k-1)(j-1)}\cdot\omega^{(k+i-2)l'+(m+j-2)k'}\Psi_{i+k-1}^{(m+j-1)}.
\end{equation}

Then we have
\begin{equation*}
\begin{aligned}
&\sum_{(k',l')\in \mathbb{Z}_{N}^{2}\setminus\{0,0\}}\omega^{k'(l'+j-1)}(T^{(l'+j)}_{k'+i}\otimes T^{(-l'+1)}_{-k'+1})\Psi_{k}^{(m)}\\
&=-\omega^{(k-1)(j-1)}\sum_{(k',l')\in \mathbb{Z}_{N}^{2}\setminus\{0,0\}}\omega^{(k-1)l'+(1-m)k'}\Psi_{i+k-1}^{(m+j-1)}\\
&=-\omega^{(k-1)(j-1)}(\sum_{(k',l')\in \mathbb{Z}_{N}^{2}}\omega^{(k-1)l'+(1-m)k'}-1)\Psi_{i+k-1}^{(m+j-1)}\\
&=-\omega^{(k-1)(j-1)}(N^2\delta_{k,1}\delta_{m,1}-1)\Psi_{i+k-1}^{(m+j-1)},
\end{aligned}
\end{equation*}
and
\begin{equation*}
\begin{aligned}
&\sum_{(k',l')\in \mathbb{Z}_{N}^{2}\setminus\{0,0\}}\omega^{l'(k'+i-1)}(T^{(l'+j)}_{k'+i}\otimes T^{(-l'+1)}_{-k'+1})\Psi_{k}^{(m)}\\
&=-\omega^{(k-1)(j-1)}\sum_{(k',l')\in \mathbb{Z}_{N}^{2}\setminus\{0,0\}}\omega^{(k+i-2)l'+(m+j-2)k'}\Psi_{i+k-1}^{(m+j-1)}\\
&=-\omega^{(k-1)(j-1)}(\sum_{(k',l')\in \mathbb{Z}_{N}^{2}}\omega^{(k+i-2)l'+(m+j-2)k'}-1)\Psi_{i+k-1}^{(m+j-1)}\\
&=-\omega^{(k-1)(j-1)}(N^2\delta_{k+i-1,1}\delta_{m+j-1,1}-1)\Psi_{i+k-1}^{(m+j-1)}.
\end{aligned}
\end{equation*}
The lemma then follows immediately.
\end{proof}
\vskip.2in

\noindent {\it Proof of Theorem \ref{th:main}}. It is an easy consequence of Eq. (\ref{eq3.12}), Lemma \ref{lemma3.5} and Lemma \ref{lemma3.6}. \hfill $\Box$

\begin{remark}
In the language of quantum computation the basis elements in the tensor product are special states with maximal
entanglement. A pure entangled state can not be factored as a tensor product of two pure states in individual factor.
It is known that they play critical role in quantum computation. In our case the maximal entangled states
are generalized Bell states. We have shown that the principal generators of the Yangian algebra act transitively
on the Bell states which means that Yangian symmetry singles out the maximally entangled states in this case
in a canonical way in terms of Hopf algebra structure and the principal generators.
\end{remark}

From Theorem  \ref{th:main}, it is easy to get the following ``subrepresentation theory" (cf. \cite{CP2}).

\begin{corollary}
With the notations as above, let $W=V(\lambda_1, a)\otimes V(\lambda_{N-1}, b)$. If $|a-b|\ne \frac{N}{2}$, then $W$ is an irreducible
representation of the Yangian $Y(\mathfrak{sl}_N)$. Otherwise, $W$ has an unique proper  $Y(\mathfrak{sl}_N)$-subrepresentation $V$ given as follows.
\begin{enumerate}
\item If $a-b=\frac{N}{2}$, then we have $V\cong V_0$ and $W/V\cong V_{\rm ad}$ as vector spaces.
\item If $a-b=-\frac{N}{2}$, then we have $V\cong V_{\rm ad}$ and $W/V\cong V_{0}$ as vector spaces.
\end{enumerate}
\end{corollary}

\subsection{The quantum number ${\bf J}^2$ and the maximal degree of entanglement} \label{sub3.3}
The Casimir operators of Lie algebras provide quantum numbers in quantum mechanics. The analogues of Casimir operators  ${\bf J}^2$ for the $J$-parts of Yangians also provide
a new kind of quantum numbers. The quantum number ${\bf J}^2$ was introduced in \cite{BGX} for $Y(\mathfrak{sl}_2)$
in the study of the spectrum of hydrogen atom. It was extensively studied and applied in several models such as few-body systems \cite{BGX2}, molecule $\{ V6\}$ and
four-spin Heisenberg chain \cite{PBG}. Most studies of ${\bf J}^2$
 have only been focused on $Y(\mathfrak{sl}_2)$ for the spin-chain. It is now natural to generalize the concept to
 the general case of $Y(\mathfrak{sl}_N)$. Based on Theorem \ref{th:main} we will give an explicit formula of ${\bf J}^2$. The new quantum number ${\bf J}^2$ is expected to be
useful in further development.

First we recall some notations. Since Lie algebras describe the angular momentums $\bf I$, we use the physical notation ${\bf I}^2$ to denote the Casimir operator. In the case
$\mathfrak{sl}_N$ with the principal basis $\{A_{ij}\}$ and its dual basis $\{\frac{\omega^{ij}}{N}A_{-i,-j}\}$, we
see that (cf. Eq.~(\ref{eq3.10}))
  \begin{equation}\label{eq:I}
{\bf I}^2=\sum_{(i,j)\in \mathbb{Z}_{N}^{2}\setminus\{0,0\}}\frac{\omega^{ij}}{N}A_{ij}A_{-i,-j} \in U(\mathfrak{sl}_N).
 \end{equation}
Correspondingly, we introduce a new operator for the Yangian  $Y(\mathfrak{sl}_N)$ as follows.
 \begin{equation}\label{eq:J}
{\bf J}^2=\sum_{(i,j)\in \mathbb{Z}_{N}^{2}\setminus\{0,0\}}\frac{\omega^{ij}}{N}J(A_{ij})J(A_{-i,-j}) \in Y(\mathfrak{sl}_N).
 \end{equation}

\begin{remark}
{Similar to the fact that the Casimir element (or ${\bf I}^2$) is independent from the choice of the bases, we
have that ${\bf J}^2$ is independent from the choice of dual bases.}
\end{remark}

\begin{lemma}\label{le:J} The quantum number ${\bf J}^2$ satisfies the following commutation relations.
$$ [{\bf I}^2, {\bf J}^2]=[{\bf J}^2, A_{ii}]=[{\bf I}^2, A_{ii}]=0 \in Y(\mathfrak{sl}_N), i=1,\cdots, N-1.$$
\end{lemma}

\begin{proof} Since all identities are shown similarly, we only give the proof for $[{\bf J}^2, A_{ii}]=0$.
For a fixed $m$ we have
\begin{eqnarray*}
A_{mm}\frac{\omega^{ij}}{N}J(A_{ij})J(A_{-i,-j})&=&\frac{\omega^{ij}}{N} (J(A_{ij})A_{mm}J(A_{-i,-j})+[A_{mm}, J(A_{ij})]J(A_{-i,-j}))\\
&=& \frac{\omega^{ij}}{N} (J(A_{ij}J(A_{-i,-j})A_{mm}+J(A_{ij})[A_{mm}, J(A_{-i,-j})]\\
&\mbox{}& +[A_{mm}, J(A_{ij})]J(A_{-i,-j})).
\end{eqnarray*}
Consequently it follows that
\begin{eqnarray*}
[A_{mm},{\bf J}^2]&=&\sum_{(i,j)\in \mathbb{Z}_{N}^{2}\setminus\{0,0\}}\frac{\omega^{ij}}{N}( J(A_{ij})J([A_{mm},J(A_{-i,-j})])+J([A_{mm},A_{ij}])J(A_{-i,-j}))\\
&=& \sum_{(i,j)\in \mathbb{Z}_{N}^{2}\setminus\{0,0\}} \frac{\omega^{ij}}{N}[ (\omega^{im}-\omega^{jm})J(A_{m+i,m+j})J(A_{-i,-j})\\
&\mbox{}& +(\omega^{-im}-\omega^{-jm})J(A_{ij})J(A_{m-i,m-j})]\\
&=&  \sum_{(i,j)\in \mathbb{Z}_{N}^{2}\setminus\{0,0\}} \frac{\omega^{ij}}{N}  (\omega^{im}-\omega^{jm})J(A_{m+i,m+j})J(A_{-i,-j})\\
&\mbox{}& + \sum_{(i,j)\in \mathbb{Z}_{N}^{2}\setminus\{0,0\}} \frac{\omega^{(i+m)(j+m)}}{N} (\omega^{-(i+m)m}-\omega^{-(j+m)m} )J(A_{m+i,m+j})J(A_{-i,-j})\\
&=&0.\end{eqnarray*}
\end{proof}

\begin{remark} The above result was proved for $Y(\mathfrak{sl}_2)$ in \cite{BGX, BGX2}.
Note that $\{A_{ii}|i=1,\cdots, N-1\}$ is a basis of the Cartan subalgebra of $Y(\mathfrak{sl}_N)$.
It is known that degenerate eigenstates under the conservation usually make things complicated.
We can take $\{{\bf
I}^2,{\bf J}^2, A_{ii}, i=1,\cdots N-1\}$ as a generalized conserved set for studying physical models.  This will help
differentiate degenerate states. For example,
in the 3 spin-$1/2$ system under the action of $\mathfrak{sl}_2$, there are two degenerate states of total spin $1/2$
with respect to $\{{\bf I}^2, A_{11}\}$. By considering the new quantum number ${\bf J}^2$, these degenerate states
become non-degenerate \cite{BGX2}. This shows that ${\bf J}^2$ can distinguish eigenstates that appear degenerate under
the usual conserved set.
\end{remark}

The following conclusion is a direct consequence of Lemma \ref{le:J}.

\begin{proposition} The quantum number ${\bf J}^2$ acts on $V_{\rm ad}$ or $V_0$ as a scalar multiplication
explicitly as follows.
\begin{equation}\label{eq:J2}
{\bf J}^2\cdot\Psi^{(m)}_{k}=\left\{
                      \begin{array}{ll}
                        (\frac{(N^2-1)(a^2+b^2)}{N}-\frac{N(N^2-1)}{4}-\frac{2ab(N^2-1)}{N})\Psi^{(m)}_{k}, & \hbox{(k,m)=(1,1);} \\
                        (\frac{(N^2-1)(a^2+b^2)}{N}-\frac{N}{4}+\frac{2ab}{N})\Psi^{(m)}_{k}, & \hbox{(k,m)$\neq$(1,1).}
                      \end{array}
                    \right.
\end{equation}
\end{proposition}

\begin{proof}
In Eqs. (\ref{fomula2.1}) and (\ref{modified principal basis}) we have seen that $T_{i}^{(j)}$
and $\frac{\omega^{(i-1)(j-1)}}{N}T_{-i+2}^{(-j+2)}$ are dual base of $\mathfrak{sl}_N$ ($(i, j)\neq (1, 1)$).
Thus,
\begin{equation}
{\bf J}^2=\sum_{(i,j)\ne (1,1)}\frac{\omega^{(i-1)(j-1)}}{N}J(T_{i}^{(j)})J(T_{-i+2}^{(-j+2)}).
\end{equation}
It follows from Theorem \ref{th:main} that
\begin{align}\nonumber
&J(T_{-i+2}^{(-j+2)})\Psi^{(m)}_{k}\\ \label{3.23}
&=[(a+\frac{N}{2}(\delta_{i,k}\delta_{j,m}-\delta_{k,1}\delta_{m,1}))\omega^{(-j+1)(k-1)}-b\omega^{(-i+1)(m-1)}
]\Psi^{(m-j+1)}_{k-i+1},\\ \nonumber
&J(T_{i}^{(j)})\Psi^{(m-j+1)}_{k-i+1}\\ \label{3.24}
&=[(a+\frac{N}{2}(\delta_{k,1}\delta_{m,1}-\delta_{k,i}\delta_{m,j}))\omega^{(j-1)(k-i)}-b\omega^{(i-1)(m-j)}
]\Psi^{(m)}_{k}.
\end{align}
Using Eqs. (\ref{3.23}) and (\ref{3.24}) we have
\begin{align*}
&\omega^{(i-1)(j-1)}J(T_{i}^{(j)})J(T_{-i+2}^{(-j+2)})\Psi^{(m)}_{k}=\omega^{(i-1)(j-1)}\times\\
&[(a+\frac{N}{2}(\delta_{i,k}\delta_{j,m}-\delta_{k,1}\delta_{m,1}))\omega^{(-j+1)(k-1)}-b\omega^{(-i+1)(m-1)}]\times\\
&[(a+\frac{N}{2}(\delta_{k,1}\delta_{m,1}-\delta_{k,i}\delta_{m,j}))\omega^{(j-1)(k-i)}-b\omega^{(i-1)(m-j)}]\Psi^{(m)}_{k}\\
&=[(a+\frac{N}{2}(\delta_{i,k}\delta_{j,m}-\delta_{k,1}\delta_{m,1}))\omega^{(j-1)(i-k)}-b\omega^{(i-1)(j-m)}]\times\\
&[(a+\frac{N}{2}(\delta_{k,1}\delta_{m,1}-\delta_{k,i}\delta_{m,j}))\omega^{(j-1)(k-i)}-b\omega^{(i-1)(m-j)}]\Psi^{(m)}_{k}\\
&=[(a^2-\frac{N^2}{4}(\delta_{i,k}\delta_{j,m}-\delta_{k,1}\delta_{m,1})^2)\Psi^{(m)}_{k}+b^2\Psi^{(m)}_{k}\\
&+(-ab+\frac{bN}{2}(\delta_{i,k}\delta_{j,m}-\delta_{k,1}\delta_{m,1}))\omega^{(j-1)(k-1)-(i-1)(m-1)}\Psi^{(m)}_{k}\\
&+(-ab-\frac{bN}{2}(\delta_{i,k}\delta_{j,m}-\delta_{k,1}\delta_{m,1}))\omega^{(i-1)(m-1)-(j-1)(k-1)}\Psi^{(m)}_{k}\\
&=(a^2+b^2)\Psi^{(m)}_{k}-\frac{N^2}{4}(\delta_{i,k}\delta_{j,m}-\delta_{k,1}\delta_{m,1})^2\Psi^{(m)}_{k}\\
&-ab(\omega^{(j-1)(k-1)-(i-1)(m-1)}+\omega^{(i-1)(m-1)-(j-1)(k-1)})\Psi^{(m)}_{k}\\
&+\frac{bN}{2}(\delta_{i,k}\delta_{j,m}-\delta_{k,1}\delta_{m,1})\omega^{(j-1)(k-1)-(i-1)(m-1)}\Psi^{(m)}_{k}\\
&-\frac{bN}{2}(\delta_{i,k}\delta_{j,m}-\delta_{k,1}\delta_{m,1})\omega^{(i-1)(m-1)-(j-1)(k-1)}\Psi^{(m)}_{k}
\end{align*}
where we have used $\sum\limits_{i=1}^N\omega^{ik}=0$ for any $k$. This completes the proof of Eq.~(\ref{eq:J2}).
\end{proof}

\begin{corollary} The the action of ${\bf J}^2$ on  $V(\lambda_1,a)\otimes V(\lambda_{N-1},b)$ is
a scalar multiplication if and only if $ab=\frac{-N^2+2}{8}$.
\end{corollary}

\begin{proof}
It follows immediately from the following identity:
$$\frac{(N^2-1)(a^2+b^2)}{N}-\frac{N(N^2-1)}{4}-\frac{2ab(N^2-1)}{N}=\frac{(N^2-1)(a^2+b^2)}{N}-\frac{N}{4}+\frac{2ab}{N}.$$
\end{proof}

\begin{remark}
{\rm If $ab=\frac{-N^2+2}{8}$, then all $N^2$ states $\Psi^{(m)}_{k}$
share the
same eigenvalue
$$\rho=\frac{(N^2-1)(a^2+b^2)}{N}-\frac{N}{2}+\frac{1}{2N}$$
of ${\bf J}^2$, that is, ${\bf J}^2 \Psi_k^{(m)}=\rho  \Psi_k^{(m)}$  for any $k,m=1,\cdots, N$. Since
$\Psi_k^{(m)}$ are of maximal degree of entanglement and all
the wave functions share the same
eigenvalue, one can imagine that the quantum number ${\bf J}^2$ selects
the maximal degree of entanglement. It coincides with the similar phenomenon for the cases $N=2, 3$ \cite{GBXZ}.
}
\end{remark}

\medskip

\centerline{\bf Acknowledgments}
NJ gratefully acknowledges the partial support of Max-Planck Institut f\"ur Mathematik in Bonn, Simons Foundation grant 198129, and NSFC grant 11271138 during this work.
CB and MG thank for the support of NSFC.

\bibliographystyle{amsalpha}

\end{document}